\documentclass{article}
\usepackage{amsthm}

\usepackage{amssymb}
\usepackage{amsfonts}
\usepackage{amsmath}
\usepackage{graphicx}
\usepackage{mathrsfs}

\newtheorem{theorem}{Theorem}

\newtheorem{corollary}[theorem]{Corollary}

\newtheorem{lemma}[theorem]{Lemma}

\begin{document}

\title{Limit theorems for pure death processes coming down from infinity   
}
\author{Serik Sagitov\footnote{Chalmers
University and University of Gothenburg, 412 96 Gothenburg, Sweden. Email address: serik@chalmers.se}  \quad and \quad  Thibaut France\footnote{\'Ecole Polytechnique, route de Saclay, 91128 Palaiseau Cedex-France; Email address: thibaut.france@polytechnique.edu}}

\maketitle

\begin{abstract}
We consider a pure death process $(Z(t), t\ge0)$ with death rates $\lambda_n$
satisfying the condition $\sum_{n=2}^\infty \lambda_n^{-1}<\infty$ of coming from infinity, $Z(0)=\infty$, down to an absorbing state $n=1$. 
We establish limit theorems for $Z(t)$ as $t\to0$, which strengthen the results that can be extracted from 
\cite{BMR}. 
We also prove a large deviation theorem assuming that $\lambda_n$ regularly vary as $n\to\infty$ with an index $ \beta>1$. It generalises a similar statement  with $\beta=2$ obtained in \cite{DPS} for 
 $\lambda_n={n\choose 2}$.
\end{abstract}

\emph{Key words}: Almost sure convergence, large deviations, Kingman's coalescent.

\section{Introduction}\label{intro}
The number of lineages in the Kingman coalescent \cite{King} instantaneously comes down from infinity
by jumps $n\to n-1$ at rate $\lambda_n={n\choose 2}$.
As a natural extention of the Kingman setting,  
we consider  a pure death process $(Z(t), t\ge0)$ with an absorbing state $n=1$, described by a sequence of death rates $(\lambda_n, n\ge 2)$ such that
\begin{equation}\label{c1}
\sum_{n=2}^\infty \lambda_n^{-1}<\infty.
\end{equation}
Assume that $Z(0)=\infty$ and denote by $T_n$ the first time when $Z(t)$ hits a given state $n\ge1$. Clearly,
$$T_n=X_{n+1}+X_{n+2}+\ldots,$$ 
where  $X_2, X_3,\ldots$ are independent exponentially distributed  holding times with  $\mathbb EX_i=\lambda_i^{-1}$. Under condition \eqref{c1}, the mean value of the hitting time $T_n$
\begin{equation}\label{me}
A_n=\mathbb ET_n=\sum_{i=n+1}^{\infty} \lambda_i^{-1}
\end{equation}
 is finite, and $A_n\to0$ as $n\to\infty$. 
 Thus the  process instantaneously comes down from infinity, in that  $\mathbb P(Z(t)=\infty|Z(0)=\infty)=0$ for any $t>0$. 
 
 In this paper we are interested in the asymptotic properties of $Z(t)$ as $t\to0$. In view of the relation $\{Z(t)> n\}=\{T_n> t\}$,  the  step function 
 $$v(t)=\sum_{n=2}^\infty n1_{[A_n,A_{n-1})}(t)+1_{[A_1,\infty)}(t),$$
being a generalised inverse of the sequence $(A_n)
$, gives  the {\it speed of coming down from infinity} for the process $Z(t)$, cf  \cite{BBL}. 
Recall that for the Kingman coalescent, $A_n=2/(n+1)$ and $v(t)\sim 2/t$ as $t\to0$. 

Our main results are presented in Sections \ref{Rgr} and \ref{MR}. Section \ref{Rgr} contains two comprehensive limit theorems. 
Theorem \ref{l1}, dealing with $T_n$, can be deduced from more general results recently obtained in \cite{BMR} for birth-death processes, however, our specialised proofs are more direct. Theorem \ref{ThZ}, dealing with $Z(t)$, improves the conditions for the laws of large numbers and the central limit theorem compared to their counterparts given in \cite{BMR}. In particular, our Theorem \ref{ThZ}  (i)  states that ${Z(t)/v(t)}\to1$ in probability as $t\to0$ under a very mild restriction
 \begin{align}
\limsup_{n\to\infty}{A_{nx}\over A_{n}}<1,\quad \mbox{for all } x>1. \label{Rxr}
\end{align}
\begin{quote}
{\it Notational agreement}: whenever in place of an integer index we put a non-integer number, say $u$, we mean that the actual index is ${\lfloor u\rfloor}$, so that  $A_{nx}:=A_{\lfloor nx\rfloor}$.
\end{quote}
In Section  \ref{uoc} we give a number of examples illustrating a wide range of possible growth patterns  covered by Theorem \ref{ThZ} for the speed function $v(t)$ as $t\to0$.
Section \ref{MR} presents an explicit  large deviation theorem generalizing a recent result in \cite{DPS} obtained for the Kingman coalescent. The remaining sections are devoted to self-contained  proofs.

Notice that our results can be also interpreted in terms of an explosive pure birth process  $N(u)=Z(T_1-u)$ obtained from the pure death process $(Z(t), 0<t\le T_1)$ by time reversing. The time-reversed process $N(u)$ can be viewed as a simple model for the number  of neutrons at time $u$  in a nuclear chain reaction  exploding at the finite random time  $T_1$, see \cite{Pa} and  \cite{WW}. 
Knowing the speed of explosion $v(t)$ and the current population size  $N(u)$, one can hope to predict the time $t=T_1-u$  left to the explosion event, cf  \cite{SS}.

\section{Limit theorems for $T_n$ and  $Z(t)$}\label{Rgr}

Recall \eqref{me} and put
\[B_n^2={\rm Var } (T_n)=\sum_{i=n+1}^{\infty} \lambda_i^{-2}, \qquad C_n^3=\sum_{i=n+1}^{\infty} \lambda_i^{-3}.\]

\begin{theorem} \label{l1}
Consider a pure death process with parameters $(\lambda_n)$ satisfying condition \eqref{c1}. 

(i) If  \begin{equation}\label{c2}
\lambda_n/\lambda_{n+1}\to\alpha\in[0,1),\quad n\to\infty,
\end{equation}
 then for each fixed $x\ge0$,
$$\mathbb P(A_n^{-1}T_n\le x)\to F_\alpha(x),\quad n\to\infty,$$ 
where the limit distribution has Laplace transform
\[\int_0^\infty e^{-ux}dF_\alpha(x)= \prod_{i\ge 0}{1\over u\alpha^i(1-\alpha)+1}. \]

(ii)  $A_n^{-1}T_n\to1$ in probability as $n\to\infty$, if
\begin{equation}
 \label{BoA}
  \lambda_n ^{-1}=o(A_n),\quad n\to\infty.
\end{equation}

(iii) $A_n^{-1}T_n\to1$  almost surely as $n\to\infty$, if
\begin{equation}
  \label{cond1}
  \sum_{i=1}^{\infty} (\lambda_{i+1}A_i)^{-2} < \infty.
\end{equation}

(iv) If \eqref{BoA} holds and furthermore
\begin{equation}
  \label{cond2}
  C_n =o(B_n),\quad n\to\infty,
\end{equation}
then for all $x\in(-\infty,\infty)$,
\[\mathbb P\Big({T_n-A_n\over B_n}\le x\Big)\to \Phi(x),\quad n\to\infty,\]
where $\Phi(x)$ is the standard normal distribution function.

\end{theorem}

\noindent{\bf Remarks 1-5}

1. Condition  \eqref{c2} implies 
\[ \lambda_{n+1}A_n\to1+\alpha+\alpha^2+\ldots={1\over1-\alpha}, \]
yielding 
\begin{equation}\label{lambdaA} 
(\lambda_{n+i}A_{n})^{-1}\to \alpha^{i-1}(1-\alpha), \quad i\ge1.
\end{equation}
It is equivalent to the  condition
\[(\lambda_{n+1}A_n)^{-1}\to\tilde \alpha:=1-\alpha\in[0,1),\]
which in \cite{BMR} is used to define the "fast regime" of coming down from infinity. The reverse part of this equivalence is seen from the recursion
\[\lambda_{n+1}A_n=1+{\lambda_{n+1}\over\lambda_{n+2}}(\lambda_{n+2}A_{n+1}).\]

2. Condition \eqref{c2} implies  $A_{n+1}/A_n\to\alpha$, and therefore,
\begin{equation}\label{a11} 
A_{nx}=o(A_n), \quad\mbox{for all } x>1.
\end{equation}

3. Condition \eqref{c1} together with 
\begin{equation}\label{a1} 
\lambda_n/\lambda_{n+1}\to1,\quad n\to\infty,
\end{equation}
imply condition  \eqref{BoA}.

4. Condition \eqref{BoA} is equivalent to
\begin{equation}
 \label{BoAn}
 B_n =o(A_n), \quad n\to\infty.
\end{equation}
To verify this, let us fix an arbitrary  $\epsilon\in(0,1)$. If \eqref{BoAn} holds, then for sufficiently large $n$,
\[   \lambda_{n+1}^{-2}\le B_n^2\le \epsilon^2 A_n^2,
\]
so that $\lambda_{n+1}^{-1}\le \epsilon(\lambda_{n+1}^{-1}+A_{n+1})$ and $\lambda_{n+1}^{-1}\le \epsilon(1-\epsilon)^{-1}A_{n+1}$, which  implies  \eqref{BoA}. On the other hand, given  \eqref{BoA},
\[  A_n^{2}-B_n^2=2\sum_{i=n+1}^{\infty}\sum_{j=i+1}^{\infty} \lambda_i^{-1}\lambda_j^{-1}=2\sum_{i=n+1}^{\infty}\lambda_i^{-1}A_i \ge \epsilon^{-1}B_n^2,
\]
for all  sufficiently large $n$, which yields \eqref{BoAn}.

5. Condition  \eqref{cond1} implies \eqref{BoAn} due to the inequality
\[  A_n^2\sum_{i=n}^{\infty} (\lambda_{i+1}A_i)^{-2} \ge B_n^2.
\]


\begin{theorem} \label{ThZ}
Consider a pure death process with parameters $(\lambda_n)$ satisfying conditions \eqref{c1} and \eqref{Rxr}.

(i) $Z(t)/v(t)\to1$  in probability as $t\to0$.

(ii) $Z(t)/v(t)\to1$ almost surely, if 
for each $\epsilon\in(0,1)$,
\begin{equation}
  \label{cond1e}
  \sum_{i=1}^{\infty} (\lambda_{i+1}A_{i(1-\epsilon)})^{-2} < \infty.
\end{equation} 

(iii) If 
 condition \eqref{c2} holds, then for each $k=0,\pm1,\pm2,\ldots$, 
$$\mathbb P(Z(A_n)\le n+k)\to F_\alpha(\alpha^{-k}),\quad n\to\infty.$$ 

(iv) Let \eqref{BoA} and \eqref{cond2} hold. If $b_n=o(n)$ is such that  for all $x\in(-\infty,\infty)$,
\[{A_n-A_{n+xb_n}\over B_{n+xb_n}}\to h(x),\quad n\to\infty,\]
then
\[\mathbb P\Big({Z(t)-v(t)\over b_{v(t)}}\le x\Big)\to \Phi(h(x)),\quad t\to0.\]

\end{theorem}

An important class of the pure death processes coming down from infinity is set out by the constraint
\begin{equation}\label{crv}
 \lambda_n= n^{\beta}L(n),\quad \beta> 1,
\end{equation}
where the function $L:[1,\infty)\to(0,\infty)$ is assumed to slowly vary at infinity. For the Kingman coalescent, this condition holds with $\beta=2$. By the properties of regularly varying functions, see \cite{BGT}, condition \eqref{crv} entails
$$A_n=n^{1-\beta}L_1(n),\quad L_1(n)\sim (\beta-1)^{-1}L^{-1}(n),\quad n\to\infty,$$
implying that $v(t)$ regularly varies at zero with index ${1\over 1-\beta}$. In this case, condition  \eqref{Rxr} holds but not  \eqref{a11}. The following statement is easily obtained from parts (ii) and (iv) of Theorem \ref{ThZ}.

\begin{corollary}\label{Co}
 If  condition \eqref{crv} holds,
 then $Z(t)/v(t)\to1$ almost surely  and the limit distribution of ${Z(t)-v(t)\over \sqrt{v(t)}}$ is normal with mean zero and variance ${1\over2\beta-1}$.

\end{corollary}

\noindent{\bf Remarks 6-9}

6. Parts (i) and (ii) of Theorem \ref{ThZ} should be compared to the pure death case of Theorems 4.3 and 4.4 in \cite{BMR}. Our laws of large numbers are stated under much weaker conditions. Notice that  \eqref{cond1} implies \eqref{cond1e}.

7. Part (iii) has no counterpart in \cite{BMR}.

8. Part (iv) should be compared to  the pure death case of Proposition 4.6 in \cite{BMR}. 

9. Corollary \ref{Co} should be compared to Theorem 5.1 in \cite{BMR}.

\section{Examples}\label{uoc}

Below we give five simple examples illustrating the wide range of regimes covered by Theorem \ref{ThZ}. For all our examples, the key condition \eqref{c1} is easily verified.
Paradoxically, the faster is the  decay of $A_n$ as $n\to\infty$, the slower is the speed of coming down from infinity.

1. Let  $A_n=(\log n)^{-a}$ for some $a>0$. Then, as $n\to\infty$,
$$\lambda_n\sim a^{-1} n(\log n)^{1+a},\quad B_n\sim a^{-1}n^{-1/2}(\log n)^{-1-a},\quad C_n\sim a^{-1}n^{-2/3}(\log n)^{-1-a}.$$
In this case  conditions \eqref{BoA}, \eqref{cond1}, \eqref{cond2},  \eqref{a1} hold, and
 \[ v(t)\sim  \exp\{t^{-1/a}\},\quad t\to0.\]
However, in this case condition \eqref{Rxr} is not valid and Theorem \ref{ThZ} can not be applied.

2. If  $A_n\sim cn^{1-\beta}$ for some $\beta>1$ and $c>0$, then condition \eqref{crv} is valid and the speed function
 \[ v(t)\sim  c^{-{1 \over \beta- 1 }}t^{{1 \over \beta- 1 }},\quad t\to0,\]
 suggests polynomial growth.
 This holds in particular, if $\lambda_n={2n\choose 3}$, with $\beta=2$. In this case, the process $2Z(t)$ describes a triple-wise coalescent  (in contrast to the pair-wise Kingman coalescent). 
 
 3.  If  $A_n=e^{-n^{\rho}}$ for some $\rho\in(0,1)$, then 
 \[ v(t)\sim  (\log t^{-1})^{1/\rho},\quad t\to0.\]
 In this case, both \eqref{a1} and \eqref{a11} are valid. Observe that for $\rho\in[{1\over 2},1)$, condition  \eqref{cond1e} holds for all $\epsilon\in(0,1)$ while condition \eqref{cond1} is not satisfied. 
 
 4. Turning to the Example 2 from Section 3.3 in \cite{BMR},  put $A_n=e^{-n/\log n}$. It was shown that in this case, $A_n^{-1}T_n\to1$ in probability, but not almost surely. For this example, the speed function has the following asymptotics
 \[ v(t)\sim  (\log t^{-1})(\log\log  t^{-1}),\quad t\to0.\]
Here, condition \eqref{a11} is satisfied together with condition  \eqref{cond1e}, thus by Theorem \ref{ThZ} (ii), we have almost sure convergence $Z(t)/v(t)\to1$.
 
5.  If  $A_n=e^{-n}$, then 
the fast decay of $A_n$ ensures that condition \eqref{c2} is satisfied with $\alpha=1/e$, and 
we have almost sure convergence $Z(t)/v(t)\to1$ with
\[ v(t)\sim  \log t^{-1},\quad t\to0.\]
For this example, condition \eqref{cond2} fails and the statements on the central limit theorem does not apply.

\section{Theorems on large deviations}\label{MR}

Consider a pure death processes whose rates   regularly vary with index $\beta>1$ satisfying condition
 \eqref{crv}.
For a given $x>0$, define $\tau=\tau(x)$ as a solution of the equation
 \begin{equation*}
     \int_{1}^{\infty}    {dy\over(\beta -1)^{-1}y^{\beta }-\tau}=x.
  \end{equation*}
  Observe that $\tau(x)$ is a strongly increasing function with 
  $$\lim_{x\to0}\tau(x)=-\infty,\quad \tau(1)=0,\quad \lim_{x\to\infty}\tau(x)=(\beta-1)^{-1}.$$
  Define two families of functions by
  \begin{align*}
 I(x)&=-(\beta-1)x\tau(x)- \ln(1-(\beta -1) \tau(x)),
\end{align*}
and $J(x)=xI(x^{\beta-1})$, which are illustrated by Figure \ref{II}. Put 
$$c(\beta)=\Big\{{ (1 -1/\beta)\pi\over\sin(\pi/\beta)}\Big\}^{\beta/(\beta-1)}.$$

\begin{lemma}\label{IJx}
The above defined functions $I(x)$ and $J(x)$
 are both non-negative and strictly convex over $x\in(0,\infty)$ with $I(1)=J(1)=0$. They satisfy the following asymptotical relations
\begin{align*}
 I(x)&\sim (\beta-1)^{-1}x,\quad x\to\infty,\\
  J(x)&\sim (\beta-1)^{-1}x^\beta,\quad x\to\infty,\\
I(x)&= c(\beta) x^{-{1\over\beta-1}}-\beta(\beta-1)^{-1}\ln x^{-1}-\ln c(\beta)-\beta +o(1), \quad x\to0,\\
 J(x)&= c(\beta) - \big(\beta \ln x + \ln c(\beta) + \beta\big)x +o(x), \quad x\to0.
 \end{align*}

\end{lemma}
 
The next large deviation theorem extends a result derived in \cite{DPS} for the Kingman coalescent. 

\begin{theorem} \label{LDT}
Consider a death process satisfying \eqref{crv} with $\beta>1$.

(i)  If $x \ge 1$, then 
  \begin{align*}
   n^ {-1 } \log  \mathbb{P}(T_n  > x A_n)& \to -I(x),\quad n\to\infty,\\
  v (t) ^ {-1 }\log  \mathbb{P}(Z(t) > xv(t)) &\to -J(x),\quad t\to0.
  \end{align*}
  
 (ii) If $0<x\le 1$, then 
\begin{align*}
    n^ {-1 } \log  \mathbb{P}(T_n  < x A_n)& \to -I(x),\quad n\to\infty,\\
  v (t) ^ {-1 }\log  \mathbb{P}(Z(t) <xv(t)) &\to -J(x),\quad t\to0.
    \end{align*}
 \end{theorem}
 
 \begin{figure}
\centering
\includegraphics[scale =0.295]{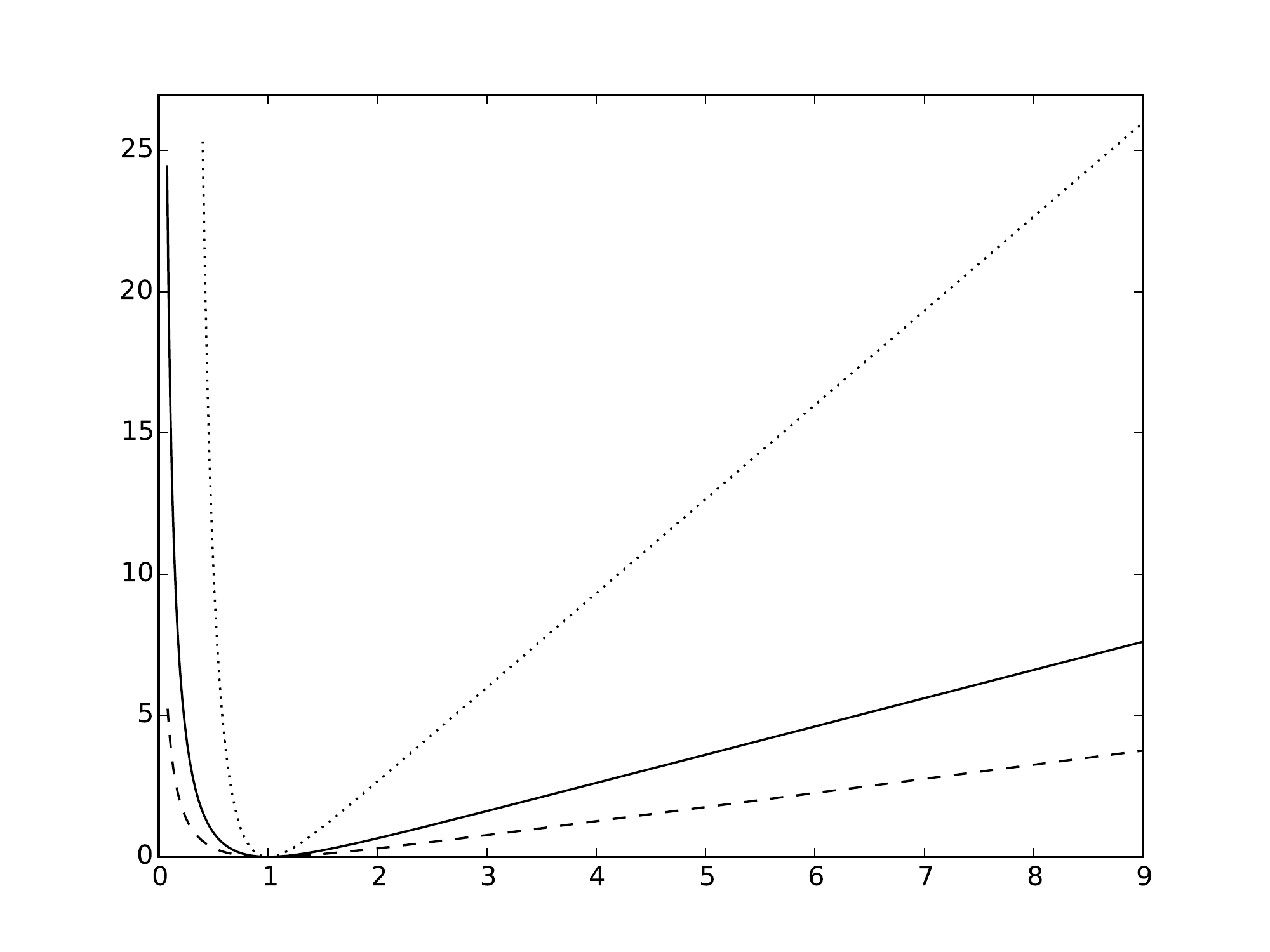}
\includegraphics[scale = 0.295]{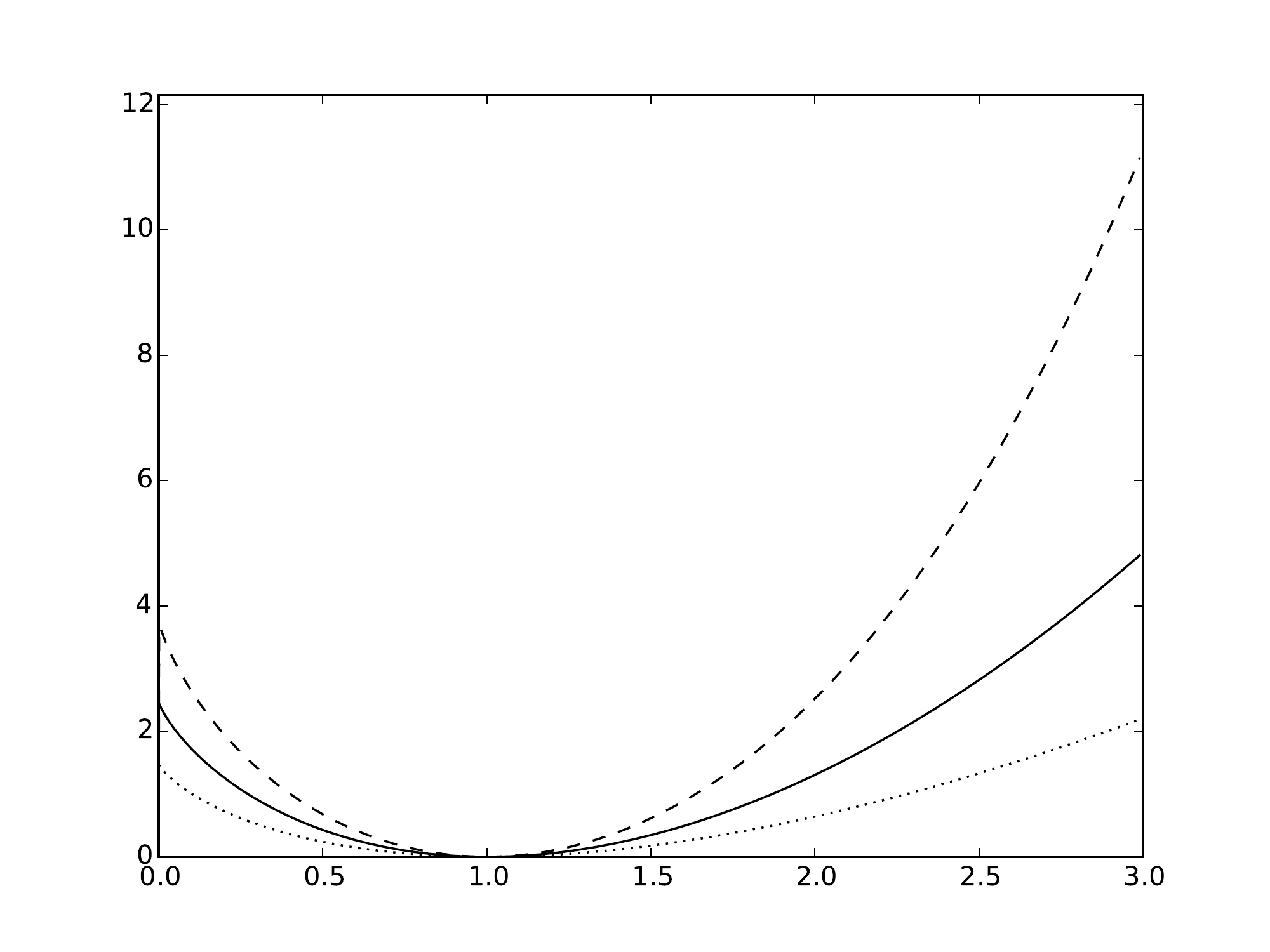}
  \caption{We use values $\beta=1.3$ (dotted lines), $\beta=2$ (solid lines), and $\beta=3$ (dashed lines) to present three pairs of profiles for the rate functions $I(x)$ on the left panel, and $J(x)$ on the right panel.
  }\label{II}
\end{figure}

\section{Proof of Theorem \ref{l1} }\label{pr1}

We start with two lemmas. Lemma  \ref{Koi} is a version of the Kolmogorov inequality, needed in the proof of Lemma  \ref{leT}.
Lemma  \ref{leT} is used in the proof of Theorem \ref{l1} (iii) and Theorem \ref{ThZ} (ii).
\begin{lemma}\label{Koi}
 If an infinite sum $\xi_1+\xi_2+\ldots $ of independent zero mean random variables converges almost surely, and $\zeta_n:=\xi_n+\xi_{n+1}+\ldots$, then for each $\epsilon>0$,
\begin{align*}
  \mathbb{P}(\sup_{k\ge n} |\zeta_k | \ge \epsilon ) &\le \epsilon^{-2}\mathbb E\zeta_n^{2},\quad n\ge1.
\end{align*}

\end{lemma}
\begin{proof}
 It is easy to check that the sequence $\zeta _n$ forms a backward martingale.
 Putting $B_k=\{|\zeta_k|\ge\epsilon,|\zeta_{k+1}|<\epsilon,|\zeta_{k+2}|<\epsilon,\ldots\}$ and using the submartingale property of $\zeta _n^{2}$ we get
\begin{align*}
 {\rm E}(\zeta^{2}_n)&\ge \sum_{k=n}^\infty {\rm E}(\zeta^{2}_n1_{B_k})\ge \sum_{k=n}^\infty  {\rm E}(\zeta^{2}_k1_{B_k})\ge \epsilon^{2}\sum_{k=n}^\infty {\rm P}(B_i) = \epsilon^{2}\mathbb{P}(\sup_{k\ge n} |\zeta_k | \ge \epsilon ).
\end{align*}

\end{proof}
 
\begin{lemma}\label{leT}
  If  \eqref{cond1e} holds for some $\epsilon\in[0,1)$,
then for any $\delta>0$,
\[\mathbb P\Big(\sup_{k\ge n}{|T_k-A_k|\over A_{k(1-\epsilon)}}>\delta\Big)\to0,\quad n\to\infty.\]

\end{lemma}

\begin{proof}
The following proof is an adaptation of the proof of  Proposition 1 in \cite{K}.  
For a given $n$, let $u_n$ be the unique natural number satisfying $$2^{-u_n-1} <A_{n(1-\epsilon)} \le2^{-u_n}.$$
Clearly $u_n\le u_{n+1}$ and $u_n\to\infty$. Putting 
$v_j=\min\{k: u_k=j\}$, we obtain 
\begin{align*}
  \mathbb{P}(\sup_{k\ge n} A_{k(1-\epsilon)}  ^{-1} | T_k - A_k |
   &\ge \epsilon ) \le \sum_{j\ge u_n}\mathbb{P}(\max_{k: u_k=j}A_{k(1-\epsilon)}  ^{-1} | T_k - A_k | \ge \epsilon )\\
 & \le\sum_{j\ge u_n}\mathbb{P}(\max_{k: u_k=j}{ |T_k - A_k| } \ge \epsilon 2^{-j-1}) \\
  &
  \le\sum_{j\ge u_n}\mathbb{P}(\sup_{k\ge v_j}{ |T_k - A_k| } \ge \epsilon 2^{-j-1}).
\end{align*}
Notice that for some $j$ the set of indices $\{k: u_k=j\}$ might be empty - in such a case the corresponding maximum is assumed to be zero.

Suppose condition \eqref{cond1e} holds for an $\epsilon\in[0,1)$. 
By Lemma \ref{Koi} applied to  
\begin{align}\label{ksi}
\xi_i=X_i-\lambda_i^{-1}
\end{align}
 having centered exponential distributions, we see that there is a positive contant $c$ such that
\begin{align*}
\sum_{j\ge u_n}\mathbb{P}(\sup_{k\ge v_j}{ |T_k - A_k| } \ge \epsilon 2^{-j-1})&\le 
\sum_{j\ge u_n}c\epsilon^{-2} 4^{j+1}\sum_{k\ge v_j}\lambda_{k+1}^{-2}\\
&=c\epsilon^{-2} \sum_{j\ge u_n}4^{j+1}\sum_{l\ge j}4^{-l}\sum_{k: u_k=l}(\lambda_{k+1}2^{-l})^{-2}\\
&\le  c\epsilon^{-2}  \sum_{l\ge u_n}\sum_{j=u_n}^l4^{j-l+1}\sum_{k: u_k=l}(\lambda_{k+1}A_{k(1-\epsilon)} )^{-2}.
\end{align*}
Thus,
\begin{align*}
 \mathbb{P}(\sup_{k\ge n} A_{k(1-\epsilon)}  ^{-1} | T_k - A_k |
&\le  4c\epsilon^{-2}  \sum_{l\ge u_n}\sum_{k: u_k=l}(\lambda_{k+1}A_{k(1-\epsilon)})^{-2}\\
&=   4c\epsilon^{-2} \sum_{k\ge K_n}(\lambda_{k+1}A_{k(1-\epsilon)})^{-2},
\end{align*}
where $K_n=\min\{k:u_k=u_n\}$ is $v_j$ for $j=u_n$. By monotonicity of $A_n$,  we have $K_n\to\infty$ as $n\to\infty$, and the statement of Lemma  \ref{leT} follows.
\end{proof}

Observe that for any given $u_0>0$, the moment generating function
\begin{align}\label{mgf}
Ee^{uT_n}&=\prod_{i=n+1}^\infty{\lambda_i\over \lambda_i-u}=
\exp\Big\{- \sum_{i=n+1}^\infty  \log (1-u\lambda_i^{-1})\Big\},\quad u\in(-\infty, u_0],
\end{align}
is well-defined for all sufficiently large $n$.\\

\noindent{\sc Proof of Theorem \ref{l1} (i)}.
By  \eqref{mgf} and  \eqref{lambdaA}, we get for each $u\ge0$,
\[\mathbb Ee^{-uT_n/A_n}=\prod_{k\ge n+1}{1\over u(\lambda_k A_n)^{-1}+1}\to \prod_{i\ge 0}{1\over u\alpha^i(1-\alpha)+1}. \]

\noindent{\sc Proof of Theorem \ref{l1} (ii)}. The stated convergence in probability  is easily derived using  the Chebyshev inequality, see Remark 4 in Section \ref{Rgr}. \\

\noindent{\sc Proof of Theorem \ref{l1} (iii)}.
The stated almost sure convergence  is a straightforward corollary of Lemma \ref{leT} with $\epsilon=0$.\\

\noindent{\sc Proof of Theorem \ref{l1} (iv)}.
Using \eqref{mgf}
and notation \eqref{ksi}, we find
\begin{align*}
\mathbb Ee^{u(\xi_n+\xi_{n+1}+\ldots)}&
=
\exp\Big\{- \sum_{i=n}^\infty u\lambda_i^{-1} + \log (1-t\lambda_i^{-1})\Big\}.
\end{align*}
Applying the Taylor formula for the logarithm we see that under condition \eqref{cond2},
\begin{align*}
Ee^{uB_n^{-1}(\xi_{n+1}+\xi_{n+2}+\ldots)}
\sim \exp\Big\{\sum_{i=n+1}^\infty{(uB_n^{-1}\lambda_i^{-1})^{2}\over 2}\Big\}= e^{u^2/2}.
\end{align*}

\section{Proof of Theorem \ref{ThZ}}
Observe that since
 \[{Z(A_{n-1})\over n}\le{Z(t)\over v(t)}\le{Z(A_{n})\over n},\quad n=v(t),\]
 convergence ${Z(t)\over v(t)}\to1$  as $t\to0$ is equivalent to ${Z(A_{n})\over n}\to1$ as $n\to\infty$.\\

\noindent{\sc Proof of part  (i)}.
Fix some arbitrary $\epsilon \in(0,1)$ and $u \in (0,\infty)$. Given \eqref{Rxr}, there exist a $\delta \in (0,1)$ and an $n_0=n_0(\epsilon,u)$ such that for all $n\ge n_0$,
\begin{align*}
&A_{n(1+\epsilon)}/A_n < \delta,\\
&A_n\lambda_k>2\delta u,\quad k>n(1+\epsilon),
\end{align*}
and the moment generating function
\[\mathbb{E} e^{uT_{n(1+\epsilon)}/A_n}=\prod_{k>n(1+\epsilon)} { 1 \over 1 - (A_n\lambda_k)^{-1}u }\]
is well defined. By Markov's inequality,
\begin{align*}
  \mathbb P(T_{n(1+\epsilon)}>A_n) 
  &\le e^{-u}\mathbb{E}e^{uT_{n(1+\epsilon)}/A_n}
 = e^{-u}\exp\Big\{ - \sum_{k>n(1+\epsilon)} \ln(1-(A_n\lambda_k)^{-1})u\Big\},
  \end{align*}
yielding
\begin{align*}
  \mathbb P(T_{n(1+\epsilon)}>A_n) 
  &\le e^{-u} \exp\Big\{\sum_{k>n(1+\epsilon)} { (A_n\lambda_k)^{-1}u \over 1-(A_n\lambda_k)^{-1}u} \Big\}\le e^{-u/2}.
 \end{align*}
 Letting $u\to\infty$, we see that $\mathbb P(T_{n(1+\epsilon)}>A_n)\to0$.
Since
\[\mathbb P(Z(A_n)>n(1+\epsilon))=\mathbb P(T_{n(1+\epsilon)}>A_n),\]
we conclude that $\mathbb P(Z(A_n)>n(1+\epsilon))\to0$.
In the same way we can prove that  $\mathbb P(Z(A_n)<n(1-\epsilon))\to0$ as $ n\to \infty.$\\

\noindent{\sc Proof of part  (ii)}. It suffices to prove that  $Z(A_n) / n\to 1$ almost surely as $n \to \infty$ or, in other terms,
\begin{align*}
& \mathbb P\Big(\sup_{k\ge n}{Z(A_k)-k\over k}>\epsilon\Big)\to0,\qquad \mathbb P\Big(\inf_{k\ge n}{Z(A_k)-k\over k}<-\epsilon\Big)\to0.
 \end{align*}
To check the first convergence, observe that
 \begin{align*}
 \mathbb P\Big\{\sup_{k\ge n}{Z(A_k)-k\over k}>\epsilon\Big\}&=\mathbb P\Big\{\exists k\ge n: Z(A_k)>(1+\epsilon)k\Big\}\\
&=\mathbb P\Big\{\exists k\ge n: T_{(1+\epsilon)k}>A_k\Big\}\\ 
&=\mathbb P\Big\{\exists k\ge n: {T_{(1+\epsilon)k}-A_{(1+\epsilon)k}\over A_k}>1-{A_{(1+\epsilon)k}\over A_k}\Big\}.
\end{align*}
It follows that by condition \eqref{Rxr}, for some $\delta\in(0,1)$ and all $n\ge n_0(\epsilon)$,
 \begin{align*}
 \mathbb P\Big\{\sup_{k\ge n}{Z(A_k)-k\over k}>\epsilon\Big\}&\le
\mathbb P\Big\{\exists k\ge n: {T_{(1+\epsilon)k}-A_{(1+\epsilon)k}\over A_k}>\delta\Big\}\\
&\le
\mathbb P\Big\{\exists k\ge n(1+\epsilon): {T_{k}-A_{k}\over A_{k/(1+\epsilon)}}>\delta\Big\},
\end{align*}
and it just remains to apply Lemma \ref{leT}. The second convergence is verified similarly.\\

\noindent{\sc Proof of part  (iii)}. The statement (iii) is an easy corollary of Theorem \ref{l1} (i) in view of Remark 2 in Section \ref{Rgr} and the relation
\[\mathbb P(Z(A_n)>n+k)=\mathbb P\Big\{{T_{n+k}\over A_{n+k}}>{A_{n}\over A_{n+k}}\Big\}.\]

\noindent{\sc Proof of part  (iv)}. The part (iv) immediately follows from Theorem \ref{l1} (iv) and equality
\begin{align*}
 \mathbb P\Big({Z(A_n)-n\over b(n)}> x\Big)&=\mathbb P(T_{n+xb(n)}> A_n)\\
 &=\mathbb P\Big({T_{n+xb(n)}-A_{n+xb(n)}\over B_{n+xb(n)}}> {A_n-A_{n+xb(n)}\over B_{n+xb(n)}}\Big).
\end{align*}

 \section{Proof of Lemma \ref{IJx}}\label{prl}

Put
  \begin{equation}
    \Lambda(u) =  -\int_{1}^{\infty} \log(1-(\beta -1)uy^{-\beta })dy,\qquad u\le1/(\beta-1),
  \end{equation}
then $\tau(x)$ satisfies $\Lambda'(\tau(x))=x$. This yields
\[\tau'(x)=1/\Lambda''(\tau(x)).\]
Integration by parts gives
\begin{align*}
  \Lambda(\tau(x)) &= - \int_1^\infty \ln(1-(\beta -1) \tau(x) y^{-\beta})dy 
  = \ln(1-(\beta -1) \tau(x)) + \beta x \tau(x).
\end{align*}
Thus  the defining expression for $I(x)$ can be rewritten as 
\[I(x)=x\tau(x)-\Lambda(\tau(x)),\quad x>0.\]
It follows that $I'(x)=\tau(x)$ and $I''(x)=\tau'(x)=1/\Lambda''(\tau(x))$. In view of
\begin{align*}
  \Lambda''(u) &= \int_1^\infty { dy \over ((\beta -1)^{-1}y^\beta -u)^2}>0,
\end{align*}
we conclude that $I(x)$ is a convex function with a minimal value $I(1)=0$.

On the other hand, 
$J(x)= xI(x^{\beta-1})$  is also a convex function with minimal value $J(1)=0$.
Indeed,
$$J'(x)=(\beta-1)x^{\beta-1}\tau(x^{\beta-1})+I(x^{\beta-1})=R(x^{\beta-1}),$$
where $R(x)=(\beta-1)x\tau(x)+I(x)$. In particular,  $J'(1)=R(1)=0$. To verify that $R'(x)>0$, observe that 
 \[R'(x)=(\beta-1)\tau(x)+(\beta-1)x\tau'(x)+\tau(x)=\beta\tau(x)+(\beta-1)x\tau'(x).\]
We have $R'(x)=\tau'(x)r(\tau(x))$, where $\tau'(x)>0$ and
\begin{align*}
 r(u)&=\beta u\Lambda''(u)+(\beta-1)\Lambda'(u)= \int_1^\infty {\beta u dy \over ((\beta -1)^{-1}y^\beta -u)^2}+\int_1^\infty {(\beta-1) dy \over (\beta -1)^{-1}y^\beta -u}\\
& = \int_1^\infty {(y^\beta +u )dy \over ((\beta -1)^{-1}y^\beta -u)^2}.
\end{align*}
Clearly, $r(u)>0$ for $u\ge-1$, and it remains to show that $r(-u)>0$ for $u>1$.  To see this, observe that in view of
\begin{align*}
r(-u)=  \int_1^\infty {(y^\beta -u )dy \over ((\beta -1)^{-1}y^\beta +u)^2}
  = u^{1/\beta-1} \int_{u^{-1/\beta}}^\infty {(y^\beta -1 )dy \over ((\beta -1)^{-1}y^\beta +1)^2},
\end{align*}
we have
\begin{align*}
r(-u)>  u^{1/\beta-1}(\beta-1)^2\int_{0}^\infty {(y^\beta -1 )dy \over (y^\beta +\beta -1)^2}=0,
\end{align*}
provided $ \beta>1$.

Turning to the stated asymptotics as $x\to \infty$, put $h = 1 - (\beta -1)u$, $z={y^\beta -1 \over h}$ and write
\begin{align*}
  \Lambda''(u) &= \int_1^\infty { dy \over ((\beta -1)^{-1}y^\beta -u)^2} = { (\beta -1)^2 \over h^2}  \int_1^\infty { dy \over ({y^\beta -1 \over h} + 1)^2} = { (\beta -1)^2 \over \beta h}  \int_0^\infty { (1+zh)^{1-1/\beta} \over (z + 1)^2} dz.
\end{align*}
This yields
\begin{align*}
  \Lambda''(u) \sim { (\beta -1)^2 \over \beta (1 - (\beta -1)u)},\quad u\to(1-\beta)^{-1}.
\end{align*}
Therefore, using L'Hospital's rule we find 
\begin{align*}
  -x^ {-1} \ln(1-(\beta -1) \tau(x)){\sim } {(\beta -1) \tau'(x)\over 1 - (\beta -1)\tau(x)}={ \beta - 1 \over \Lambda''(\tau(x)) (1 - (\beta -1)\tau(x))} \to { \beta  \over \beta -1},\quad x\to\infty.
\end{align*}
This implies
\begin{align*}
  x^ {-1} I(x)   &= - (\beta -1)\tau(x)   -x^ {-1} \ln(1-(\beta -1) \tau(x))\to -1+ { \beta  \over \beta -1}= { 1  \over \beta -1}.
  \end{align*}

The last assertion of the lemma gives an asymptotics as $x\to 0$. We prove it by first noticing that as $u\to\infty$,
  \begin{align*}
    \Lambda'(-u) 
    &=u^{1/\beta-1} \int_{u^{-1/\beta}}^{\infty}    {dy\over(\beta -1)^{-1}y^{\beta }+1} \\
   &  = u^{1/\beta-1} \int_{0}^{\infty}    {dy\over(\beta -1)^{-1}y^{\beta }+1}-u^{1/\beta-1} \int_0^{u^{-1/\beta}}   {dy\over(\beta -1)^{-1}y^{\beta }+1} \\
   &= u^{1/\beta-1} (\beta -1)^{1/\beta}  {\pi/\beta\over\sin(\pi/\beta)}-u^{-1}+o(u^{-1}),
  \end{align*}
and therefore, as $x\to 0$,
\[   x= |\tau(x)|^{1/\beta-1} (\beta -1)^{1/\beta}  {\pi/\beta\over\sin(\pi/\beta)}-|\tau(x)|^{-1}+o(|\tau(x)|^{-1}).
\]
Solving the last equation, we get as a first approximation
 \begin{equation*}
\tau(x)\sim -b(\beta) x^{-{\beta\over\beta-1}},\quad b(\beta) :=     (\beta -1)^{1\over\beta-1}  \Big({\pi/\beta\over\sin(\pi/\beta)}\Big)^{\beta\over\beta-1}=c(\beta)/(\beta-1),
  \end{equation*}
  and then more exactly
  \begin{equation*}
\tau(x)
= -b(\beta) x^{-{\beta\over\beta-1}}+x^{-1}+o(x^{-1}).
  \end{equation*}
 Thus
  \[ I(x)=-(\beta-1)x\tau(x)- \ln(1-(\beta -1) \tau(x)) = c(\beta) x^{-{1\over\beta-1}} - {\beta\over\beta-1}\ln x^{-1} -\ln c(\beta) -\beta +o(1).
\]
and as stated
$$J(x) = c(\beta) - ( \beta \ln x +\ln c(\beta) + \beta )x +o(x).$$

\section{Proof of Theorem \ref{LDT}}\label{pr5}
Here we prove only the statement (i), the statement (ii) is proved similarly. Our proof of (i) is more direct than that of \cite{DPS} and uses the classical Cramer's device of 'tilted distributions'.
 
Let $x>1$. The required upper bound for (i) is obtained from 
    \begin{align*}
      \mathbb{P}(T_n > xA_n) &= \mathbb{P}(e^{\tau(x) n A_n^{-1} T_n }>e^{x\tau(x)n} )\le  \mathbb{E}e^{\tau(x) n A_n^{-1} T_n }e^{-x\tau(x) n}.    \end{align*}
Indeed, using \eqref{mgf}  we find
    \begin{align*}
 n^{-1}  \log   \mathbb{P}(T_n > xA_n) &\le   - {1 \over n}\sum_{i=n}^{\infty} \log\Big(1-{\tau(x)\over \lambda_iA_n n^{-1}}\Big)-x\tau(x),   
  \end{align*}
and it remains to see that by the dominated convergence theorem, 
    \begin{equation*}
      - {1 \over n}\sum_{i=n}^{\infty} \log\Big(1-{u \over \lambda_iA_n n^{-1}}\Big)= \int_1^{\infty} \log\Big(1-
      {u \over \lambda_{ yn } A_n n^{-1}}\Big)dy\to\Lambda(u).
    \end{equation*}
    Here the dominating function is found from the uniform bounds
    $$ (1-\epsilon)y^{\beta-\epsilon}\le{ \lambda_{yn}\over  \lambda_n}\le (1+\epsilon)y^{\beta+\epsilon},\qquad n\ge n_0(\epsilon),\qquad y\in[1,\infty),$$
where $n_0(\epsilon)$ does not depend on $y$.
  
  The lower bound for (i) is derived using the so-called tilted distributions. For the random variables $X_i$ with exponential distributions 
  Exp$(\lambda_i)$ we introduce tilted versions $\tilde X_{i,n}$ with exponential distributions Exp$(\tilde \lambda_{i,n})$, where 
  \[\tilde \lambda_{i,n}=\lambda_i-\tau(x) nA_n^{-1},\qquad i\ge n.\]
  The parameters $(\tilde \lambda_{i,n})_{i\ge n}$ are positive for all sufficiently large $n$.
  Let $F_n(y)$ and $\tilde F_n(y)$ be the distributions functions for $T_n=\sum_{i=n+1}^\infty X_i$ and  $\tilde T_n=\sum_{i=n+1}^\infty \tilde X_{i,n}$ respectively.
We have
  \begin{align*}
    \int_{-\infty}^{\infty} e^{uy}d\widetilde{F}_n(y) &= \mathbb{E}e^{u\tilde{T}_n}=\prod_{i=n+1}^\infty {\tilde \lambda_{i,n}\over \tilde \lambda_{i,n}-u}\\
    &= {\mathbb{E}e^{(u+\tau(x)n/A_n) T_n} \over \mathbb{E}e^{(\tau(x)n/A_n)T_n}}= {1 \over \mathbb{E}e^{\tau(x)nT_n/A_n}} \int_{-\infty}^{\infty} e^{(u+\tau(x)n/A_n)y}dF_n(y) 
  \end{align*}
  implying
  \begin{equation*}
    d\widetilde{F}_n(y) = { e^{\tau(x)n/A_n y} \over \mathbb{E}e^{\tau(x)n T_n/A_n}}dF_n(y).
  \end{equation*}
  Thus, for any $b>x$, we get
  \begin{align*}
    \mathbb{P}(T_n > x A_n) &= \int_{xA_n}^{\infty} dF_n(y) \ge \mathbb{E}[e^{\tau(x)n/A_n T_n}] e^{-\tau(x)n b } \int_{xA_n}^{bA_n} d\widetilde{F}_n(y) \\
    &= \mathbb{E}[e^{{\tau(x)n \over A_n} T_n}] e^{-\tau(x)n b} \int_{xA_n}^{bA_n} d\widetilde{F}_n(y).
  \end{align*}
  
  By the dominated convergence theorem,
    \begin{align*}
  A_n^{-1}  \mathbb{E}\widetilde{T}_n &= n^{-1}\sum_{i=n+1}^{\infty} { 1 \over \lambda_i A_n n^{-1}- \tau(x) } 
  = \int_1^{\infty}  { dy \over \lambda_{y(n+1)} A_n n^{-1}- \tau(x) }\\
  & \to  \int_1^{\infty} {1 \over y^\beta (\beta - 1)^{-1}- \tau(x) }dy =\Lambda'(\tau(x))=x.
  \end{align*}
Similarly,
    \begin{align*}
      {n  \over A_n^2}\widetilde{B}^2_n &\to \int_1^{\infty} { 1 \over (y^\beta (\beta-1)^{-1} - \tau(x) )^2}dy =  \Lambda''(\tau(x)),
    \end{align*}
    and
    \begin{align*}
      {n^2  \over A_n^3}\widetilde{C}^3_n &\to \int_1^{\infty} { 1 \over (y^\beta (\beta-1)^{-1} - \tau(x) )^3}dy =  {1\over 2}\Lambda'''(\tau(x)).
    \end{align*}
  We see that $\tilde C_n=o(\tilde B_n)$ so that the distribution of ${\sqrt n(\widetilde{T}_n-x A_n)\over A_n}$ is approximately normal with zero mean and variance $\Lambda''(\tau(x))$. Thus
  \[ \int_{xA_n}^{bA_n} d\widetilde{F}^n(y)\to1/2,\]
  and we get
  \begin{align*}
      \liminf_{n\to\infty}    n^{-1} \log\mathbb{P}(T_n > xA_n)&\ge 
           \Lambda(\tau(x))- b\tau(x).
    \end{align*}
    To finish the proof of  the first part of (i) we send $b\to x$.
    
    Turning to the second part of (i) it suffices to observe that as $n\to\infty$
 \begin{equation*}
n^{-1} \log  \mathbb{P}(Z(A_n) >nx)\sim x   (nx)^{-1} \log  \mathbb{P}(T_{nx}  > x^{\beta-1} A_{nx})\to xI(x^{\beta-1}).
  \end{equation*}


\begin{thebibliography}{99}
  
\bibitem{BMR}  Bansaye, V., M\'el\'eard, S., and Richard, M.: \emph{Speed of coming down from infinity for birth and death processes}, arXiv:1504.08160

\bibitem{BBL}  Berestycki, J.,  Berestycki, N., Limic, V.: \emph{The $\Lambda$-coalescent speed of coming down from infinity.} Ann. Probab. \textbf{38} (2010) 207--233.
  
\bibitem{BGT} Bingham, N. H., Goldie, C. M., and Teugels, J. L.: \emph{Regular Variation}. Encyclopedia of mathematics and its Applications. Cambridge University Press, Cambridge, 1987.

\bibitem{DPS} Depperschmidt, A., Pfaffelhuber, P., and Scheuringer, A.: \emph{Some large deviations in Kingman's coalescent}. Electron. Commun. Probab. \textbf{20} (2015) 1-14.

 \bibitem{King} Kingman, J. F. C.: \emph{The coalescent}. Stochastic Process. Appl. \textbf{13} (1982) 235-248.
  
\bibitem{K} Klesov, O. I.: \emph{Rate of convergence of series of random variables}. Ukrainian Math. Journal. \textbf{35} (1983) 309-314.

\bibitem{Pa} Pakes, A. G.: \emph{Divergence rates for explosive birth processes}. Stochastic Process. Appl. \textbf{41} (1992) 91-99.
  
\bibitem{SS}  Sagitov, S.: \emph{On an explosive branching process}. Theory Probab. Appl. \textbf{40} (1996) 575-577.

\bibitem{WW}  Waugh, W. A. O'N.: \emph{Modes of Growth of Counting Processes with Increasing Arrival Rates}. Journal Appl. Probab. \textbf{11} (1974) 237-247.

\end{thebibliography}
\end{document}